\documentclass{article}

\usepackage{amsmath}
\usepackage{amssymb}
\usepackage{amsfonts}
\usepackage{amsthm}
\usepackage{color}
\usepackage{graphicx}
\usepackage{hyperref}
\usepackage{mathabx} 
\usepackage{extarrows}
\theoremstyle{plain}
\newtheorem{thm}{Theorem}[section]
\newtheorem{lem}[thm]{Lemma}
\newtheorem{prop}[thm]{Proposition}

\theoremstyle{definition}

\theoremstyle{remark}

\def\dgr{\texttt{dgr}}
\def\dgs{\texttt{dgs}}
\def\mod{\text{mod}\ }

\makeatletter 
\@addtoreset{equation}{section}
\makeatother  

\begin{document}

\title {\bf An equation about sum of primes with digital sum constraints}
\author{\it Haifeng Xu\thanks{\url{mailto:hfxu@yzu.edu.cn} The work is partially supported by the University Science Research Project of Jiangsu Province (14KJB110027) and the Foundation of Yangzhou University 2014CXJ004.}
}
\date{\small\today}

\maketitle

\begin{abstract}
We know that any prime number of form $4s+1$ can be written as a sum of two perfect square numbers. As a consequence of Goldbach's weak conjecture, any number great than $10$ can be represented as a sum of four primes. We are motivated to consider an equation with some constraints about digital sum for the four primes. And we conclude that the square root of the digital sum of the four primes will greater than $4$ and will not be a multiple of $3$ if the equation has solutions. In the proof, we give the method of determining whether a number is a perfect square.
\end{abstract}

\noindent{\small {\bf MSC2010:} 11A41.
}

\section{Introduction}
The Goldbach's weak conjecture \cite{weak-Goldbach-conjecture} says that every odd number greater than 5 can be expressed as the sum of three primes. (A prime may be used more than once in the same sum.)

In 1923, Hardy and Littlewood showed that, assuming the generalized Riemann hypothesis, the odd Goldbach conjecture is true for all sufficiently large odd numbers.

In 1937, Ivan Matveevich Vinogradov eliminated the dependency on the generalised Riemann hypothesis and proved directly that all sufficiently large odd numbers can be represented as a sum of three primes.

In 2013, Harald Helfgott proved the Goldbach's weak conjecture.

As a consequence of Goldbach's weak conjecture, any even number $N$ great than $10$ can be represented as a sum of four primes. Since any prime like $4s+1$ can be represented as a sum of two squares, we are motivated to consider the following equations.

Let $x_1, x_2, x_3, x_4$ be primes which satisfy the following condistions:

\begin{equation}\label{eqn:main}
\begin{cases}
x_1+x_2+x_3+x_4=a^2,\\
\sum_{i=1}^{4}\dgs(x_i)=b^2,\\
\dgs(a^2)+\dgs(b^2)=\dgs(a^2+b^2),\\
a^2+b^2 \ \text{is a prime of the form}\ 4s+1,\\
\end{cases}
\end{equation}
where $a,b$ and $s$ are positive integers. And $\dgs(n)$ is the sum of all digits of the integer $n$ in its decimal representation.

We may ask the following questions.
\begin{itemize}
\item(i)  Are there infinitely many solutions?
\item(ii) Which values will $b$ take?
\end{itemize}

In this article, we prove that $b>4$ and will not be a multiple of $3$.

\section{Some computation results}

Here we only list one of the solutions of $\{x_1,x_2,x_3,x_4\}$ for each pair $(a,b)$. We use $p_j$ to stand for the $j$-th prime number.

\begin{figure}[htbp]
  \centering
  \begin{tabular}{|r|r|r|r|r|c|}
     \hline
     $x_1$  &  $x_2$   &  $x_3$   &  $x_4$   &  $a^2 + b^2 = 4s+1$ & (a,b)\\
     \hline
  2  & 29  & 997  & 1997  & 3025 + 64 = 3089 & (55, 8)\\
  3  &  3  &  11  & 3347  & 3364 + 25 = 3389 & (58, 5)\\
  3  &  3  & 101  & 4517  & 4624 + 25 = 4649 & (68, 5)\\
  3  &  3  & 107  &  911  & 1024 + 25 = 1049 & (32, 5)\\
  3  &  3  &4217  & 7877  &12100 + 49 =12149 & (110, 7)\\
  3  & 19  & 179  &  199  &  400 + 49 =  449 & (20, 7)\\
 11  & 11  &5521  & 7001  &12544 + 25 =12569 & (112, 5)\\
     \hline
   \end{tabular}
  \caption{ $x_i\in\{p_1,p_2,\ldots, p_{1000}\}$}
  \label{table:solutions1}
\end{figure}

\begin{figure}[htbp]
  \centering
  \begin{tabular}{|r|r|r|r|r|c|}
     \hline
     $x_1$  &  $x_2$   &  $x_3$   &  $x_4$   &  $a^2 + b^2 = 4s+1$ & (a,b)\\
     \hline
 8101 & 10301 & 10301 & 12101 &  40804 + 25 = 40829 & (202, 5)\\
11003 & 17033 & 17123 & 17341 & 62500 + 49 = 62549 & (250, 7)\\
     \hline
   \end{tabular}
  \caption{ $x_i\in\{p_{1001},\ldots, p_{2000}\}$}
  \label{table:solutions2}
\end{figure}

\begin{figure}[htbp]
  \centering
  \begin{tabular}{|r|r|r|r|r|c|}
     \hline
     $x_1$  &  $x_2$   &  $x_3$   &  $x_4$   &  $a^2 + b^2 = 4s+1$ & (a,b)\\
     \hline
17393 & 19889 & 19979 & 26839 & 84100 + 121 = 84221 & (290, 11)\\
24799 & 26879 & 27299 & 27299 &106276 + 121 = 106397& (326,11)\\
     \hline
   \end{tabular}
  \caption{ $x_i\in\{p_{2001},\ldots, p_{3000}\}$}
  \label{table:solutions3}
\end{figure}

\section{Preliminary}

\begin{lem}\label{lem:1}
A square number must be one of the following forms: $16n$, $16n+1$, $16n+4$, $16n+9$.
\end{lem}

\begin{lem}\label{lem:2}
\[
\dgs(m)+\dgs(n)\equiv\dgs(m+n)(\text{mod}\ 9).
\]
\end{lem}
\begin{proof}
We explain it by example. For the equation $29+96=125$, we have
\[
\dgs(29)+\dgs(96)=\dgs(29+96)+9\cdot 2,
\]
where there are two carry flag in the sum $29+96$. For each carry flag, we should subtract $9$ from the sum of $\dgs(29)+\dgs(96)$ then the result equals $\dgs(29+96)$.

Thus, it also provide an algorithm of addition. For two positive integers $m$ and $n$. Without loss of generality, we can assume they have both $r$ digits in the decimal representations. (It means, one of the digits $m_r$ or $n_r$ may be zero.)
\[
\begin{aligned}
m&=m_{r}m_{r-1}\cdots m_2 m_1 m_0,\\
n&=n_{r}n_{r-1}\cdots n_2 n_1 n_0,\\
\end{aligned}
\]
If $m_{i}+n_{i}\geqslant 10$ for some $i=k_1,\ldots,k_s$. Then
\[
\dgs(m)+\dgs(n)=\dgs(m+n)+9\cdot s.
\]
\end{proof}

Let $\dgr(m)$ be the digit root(i.e., repeated digital sum) of number $m$. It can also be defined as
\[
\dgr(m)=1+((m-1)\mod 9).
\]

Then we have
\begin{lem}[\cite{digit_root}]\label{eqn:dgr-formula}
(1) $\dgr(m+n)=\dgr(\dgr(m)+\dgr(n))$;\\
(2) $\dgr(n_1+n_2+\cdots+n_m)=\dgr(\dgr(n_1)+\dgr(n_2)+\cdots+\dgr(n_m))$;\\
(3) $\dgr(m\cdot n)=\dgr(\dgr(m)\cdot\dgr(n))$;\\
(4) $\dgr(n^k)=\dgr(\dgr^k(n))$.
\end{lem}
\begin{proof}
(1) The first identity will be inferred by the definition of the digital root.

(2) We prove the second as an example in the case of $m=3$. The general formula will be proved by induction. By (1), we have
\[
\begin{split}
\dgr(n_1+n_2+n_3)&=\dgr(\dgr(n_1)+\dgr(n_2+n_3))\\
&=\dgr\bigl(\dgr(n_1)+\dgr(\dgr(n_2)+\dgr(n_3))\bigr)\\
&=\dgr\bigl(\dgr(\dgr(n_1))+\dgr(\dgr(n_2)+\dgr(n_3))\bigr)\\
&=\dgr\bigl(\dgr(n_1)+\dgr(n_2)+\dgr(n_3)\bigr).
\end{split}
\]

(3)
\[
\begin{split}
\dgr(m\cdot n)&=\dgr(\sum_{i=1}^{m}n)\\
&\xlongequal[]{(2)}\dgr(\sum_{i=1}^{m}\dgr(n))\\
&=\dgr(\dgr(n)\cdot m)\\
&=\dgr(\sum_{j=1}^{\dgr(n)}m)\\
&\xlongequal[]{(2)}\dgr(\sum_{j=1}^{\dgr(n)}\dgr(m))\\
&=\dgr(\dgr(m)\cdot\dgr(n)).
\end{split}
\]

(4) For $n=3$,
\[
\begin{split}
\dgr(n^3)&=\dgr(n\cdot n^2)\\
&=\dgr(\sum_{i=1}^{n}\dgr(n^2))\\
&=\dgr(n\cdot\dgr(n^2))\\
&=\dgr(\dgr(n^2)\cdot\dgr(n))\\
&\xlongequal[]{(3)}\dgr\bigl(\dgr(n)\cdot\dgr(\dgr^2(n))\bigr),\\
\end{split}
\]
on the other hand,
\[
\begin{split}
\dgr(\dgr^3(n))&=\dgr(\dgr(n)\cdot\dgr^2(n))\\
&=\dgr(\sum_{i=1}^{\dgr(n)}\dgr^2(n))\\
&=\dgr(\sum_{i=1}^{\dgr(n)}\dgr(\dgr^2(n)))\\
&=\dgr\bigl(\dgr(n)\cdot\dgr(\dgr^2(n))\bigr).
\end{split}
\]
Hence,
\[
\dgr(n^3)=\dgr(\dgr^3(n)).
\]
Suppose the forth formula in Lemma \ref{eqn:dgr-formula} is true for $k-1$. Then
\[
\begin{split}
\dgr(n^k)&=\dgr(n\cdot n^{k-1})=\dgr(n\cdot\dgr(n^{k-1}))\\
&=\dgr(\dgr(n^{k-1})\cdot\dgr(n))\\
&=\dgr\bigl(\dgr(n)\cdot\dgr(\dgr^{k-1}(n))\bigr).
\end{split}
\]
On the other hand,
\[
\dgr(\dgr^k(n))=\dgr(\dgr(n)\cdot\dgr^{k-1}(n))=\dgr\bigl(\dgr(n)\cdot\dgr(\dgr^{k-1}(n))\bigr).
\]
Therefore, $\dgr(n^k)=\dgr(\dgr^k(n))$.
\end{proof}

Here are some simple observations.\\
(1) $\dgs(p)\geqslant 2$ holds for all prime numbers $p$.\\
(2) If $p$ is a prime and $\dgs(p)=3$, then $p=3$.

\section{main results}

\begin{prop}
$b$ is great than $3$.
\end{prop}
\begin{proof}
First, by observation above, $\sum_{i=1}^{4}\dgs(x_i)\geqslant 8$. Thus we have $b>2$. Suppose there exist four primes $x_1\leqslant x_2\leqslant x_3\leqslant x_4$ such that
\[
\begin{cases}
x_1+x_2+x_3+x_4=a^2,\\
\sum_{i=1}^{4}\dgs(x_i)=3^2,\\
a^2+9\ \text{is a prime like the form }\ 4s+1.
\end{cases}
\]
If $x_1=2$, then $a^2$ is odd and thus $a^2+9$ will not be the form $4s+1$. Hence, we assume $x_1\geqslant 3$. Since $\sum_{i=1}^{4}\dgs(x_i)=3^2$, $x_1$ can not be taken as $5$ or $7$. Note that $\sum_2^4 \dgs(x_i)\geqslant 6$, we have $\dgs(x_1)\leqslant 3$. If $\dgs(x_1)=2$, then $x_1=11$ or $x_1$ is a prime like $10^k+1$. Thus $\sum_{i=2}^{4}\dgs(x_i)=7$. It infers that $x_2,x_3,x_4$ must be primes like $10\ldots 01$, i.e., $10^k+1$ or $10^k+10^h+1$. But $10^k+10^h+1$ is always a composite number. Hence, the only possible case is that $x_1=3$.

In the case of $x_1=3$, $x_2,x_3,x_4$ must be primes like $10^k+1$. It is easy to see that if $10^k+1$ is prime then $k$ must be a power of $2$. There are no primes of the form $10^k+1$ below $10^{16777216}+1=10^{2^{24}}+1$. There is a discussion about the numbers like $10^k+1$ in the physics forum \cite{physicsforums}.

Suppose there exist such primes like $10^{k}+1$. We consider the following equations:
\[
\begin{aligned}
3+11+(10^k+1)+(10^h+1)=a^2,\\
3+(10^r+1)+(10^s+1)+(10^t+1)=a^2,\\
\end{aligned}
\]
where $1<k\leqslant h$, $1<r\leqslant s\leqslant t$ are all positive integers.

For the first equation, $a^2+9=10^{k}(10^{h-k}+1)+25$, which is not a prime. For the second equation, $a^2=10^r+10^s+10^t+6$. It is not a square number for $r>1$.

Therefore, $b>3$.

\end{proof}

By using digit sum, we will have a more simpler proof and get a general result.
\begin{prop}
$b$ can not be a multiple of $3$.
\end{prop}
\begin{proof}
Let $b=3k$ in \eqref{eqn:main}, then
\[
\begin{cases}
x_1+x_2+x_3+x_4=a^2,\\
\dgs(x_1)+\dgs(x_2)+\dgs(x_3)+\dgs(x_4)=9k^2.\\
\end{cases}
\]
By Lemma \ref{lem:2},
\[
\begin{split}
&\dgs(x_1)+\dgs(x_2)+\dgs(x_3)+\dgs(x_4)\\
=&\dgs(x_1+x_2)+9u+\dgs(x_3+x_4)+9v\\
=&\dgs(x_1+x_2+x_3+x_4)+9w+9u+9v,\\
\end{split}
\]
where $u,v,w$ are integers. Hence, $9k^2=\dgs(a^2)+9K$ for some integer $K$. It infers that $9|(\dgs(a^2)+\dgs(b^2))$. Hence $a^2+b^2$ is a multiples of $9$ since $9|\dgs(a^2+b^2)$. It is a contradiction.
\end{proof}

\begin{prop}
$b$ doesn't equal $4$.
\end{prop}
\begin{proof}
If $b=4$, then $a$ must be odd. Then $x_1=2$. The system \eqref{eqn:main} takes the following form
\[
\begin{cases}
2+x_2+x_3+x_4=a^2,\\
2+\dgs(x_2)+\dgs(x_3)+\dgs(x_4)=4^2,\\
\dgs(a^2+16)=\dgs(a^2)+\dgs(16),\\
a^2+16\ \text{is a prime}.
\end{cases}
\]
It infers that
\[
\begin{cases}
x_2+x_3+x_4=a^2-2,\\
\dgs(x_2)+\dgs(x_3)+\dgs(x_4)=14,\\
\dgs(a^2+16)=\dgs(a^2)+7,\\
a^2+16\ \text{is a prime}.
\end{cases}
\]
By Lemma \ref{lem:2}, $\dgs(x_2+x_3+x_4)+9k=14$ for some integer $k$.

(1) If $k=0$, then $\dgs(x_2+x_3+x_4)=14=\dgs(a^2-2)$. From the equation $\dgs(a^2-2+18)=\dgs(a^2+16)=\dgs(a^2)+7$, we have
\[
\begin{split}
&\dgs(a^2-2)+\dgs(18)-9\ell=\dgs(a^2)+7\\
\Rightarrow&14+9-9\ell=\dgs(a^2)+7\\
\Rightarrow&\dgs(a^2)=16-9\ell.
\end{split}
\]
If $\ell=0$, then
\begin{equation}\label{eqn:case1}
\dgs(a^2)=16,\quad\dgs(a^2-2)=14.
\end{equation}
If $\ell=1$, then
\begin{equation}\label{eqn:case2}
\dgs(a^2)=7,\quad\dgs(a^2-2)=14.
\end{equation}

(2) If $k=1$, then $\dgs(x_2+x_3+x_4)=5=\dgs(a^2-2)$. Then,
\[
\begin{split}
&\dgs(a^2-2)+\dgs(18)-9\ell=\dgs(a^2)+7\\
\Rightarrow&5+9-9\ell=\dgs(a^2)+7\\
\Rightarrow&\dgs(a^2)=7-9\ell.
\end{split}
\]
In this case, $\ell=0$. Thus we have
\begin{equation}\label{eqn:case3}
\dgs(a^2)=7,\quad\dgs(a^2-2)=5.
\end{equation}

Note that
\[
\begin{split}
& \dgs(a^2)=\dgs(a^2-2+2)=\dgs(a^2-2)+\dgs(2)-9s\\
\Rightarrow & \dgs(a^2)=\dgs(a^2-2)+2-9s\\
\end{split}
\]
For case \eqref{eqn:case1}, we have $16=14+2-9s$. That is $s=0$, which means the addition $(a^2-2)+2$ has no carry.

For case \eqref{eqn:case2}, we have $7=14+2-9s$. That is $s=1$, which means the addition $(a^2-2)+2$ has a carry.

For case \eqref{eqn:case3}, we have $7=5+2-9s$. That is $s=0$, which means the addition $(a^2-2)+2$ has no carry.

{\bf Case \eqref{eqn:case1}.} Since $a$ is an odd number, $a^2$ must be end with $1,9$ or $25$. In this case, $(a^2-2)+2$ has no carry, hence we infer that $a^2$ will not end with $1$. Moreover, we assume $a^2+16$ is a prime number. Therefore $a^2$ should be the form of $100N+25$. If so, $a^2+16=100N+41$. Then we have
\[
\dgs(a^2+16)=\dgs(N)+5<\dgs(a^2)+7.
\]
Hence, it violate the restriction.

{\bf Case \eqref{eqn:case1}.} Since $a$ is an odd number, $a^2$ must be end with $1,9$ or $25$. In this case, $\dgs(a^2)=7$, thus $a^2$ will not end with $9$. Also $(a^2-2)+2$ has no carry, hence we infer that $a^2$ will not end with $1$. Therefore $a^2$ should be the form of $100N+25$. Similarly as above, $a^2+16=100N+41$. Then we have
\[
\dgs(a^2+16)=\dgs(N)+5<\dgs(a^2)+7,
\]
which violate the restriction.

{\bf Case \eqref{eqn:case2}.}

If the addition $(a^2-2)+2$ has one carry, in case \eqref{eqn:case2}, then the unit's digit of $a^2-2$ is $9$. (Remember that $a$ is odd.) Thus the unit's digit of $a^2$ is $1$. And the ten's digit must be even. Moreover, it cannot be $0$ or $8$ for the reason $\dgs(a^2-2)=14$ and $\dgs(a^2)=7$ respectively.

Since $61$ is not a square number, we only need to consider $a^2$ with $2$ or $4$ as the ten's digit. Also note that $a^2=8N+1$. Hence it looks like
\[
\begin{aligned}
10\cdots 010\cdots 041,\\
20\cdots 041,\\
30\cdots 0121,\\
10\cdots 0321,\\
10\cdots 020\cdots 0121,\\
20\cdots 010\cdots 0121,\\
10\cdots 010\cdots 010\cdots 0121.\\
\end{aligned}
\]
The lemmas in the next subsection will show that such numbers are all not perfect squares. Therefore, we conclude that $b$ doesn't equal $4$.
\end{proof}

\subsection{Some lemmas about non-square numbers}

\begin{lem}
Assume the number $N=a_n a_{n-1}\cdots a_2a_1a_0$ with 41 as the last two digits, and $a_2$ is odd. Then $N$ is not a square number. If $N=a_n a_{n-1}\cdots a_2a_1a_0$ with 21 as the last two digits, and $a_2$ is even. Then $N$ is not a square number.
\end{lem}
\begin{proof}
Note that $1000K+141\not\equiv 1\pmod 8$, $1000K+21\not\equiv 1\pmod 8$.
\end{proof}

\begin{lem}\label{lem:41}
The number $10^n+10^k+41$ is not a square number for any $n$ and $k$.
\end{lem}
\begin{proof}
\[
\begin{aligned}
10^n &\equiv (-1)^n\quad(\mod 11),\\
10^k &\equiv (-1)^k\quad(\mod 11),\\
  41 &\equiv 8\quad(\mod 11).
\end{aligned}
\]
Hence, if $10^n+10^k+41\equiv q(\mod 11)$, then $q$ is one of the three numbers: $10,8,6$. But the remainders by $11$ of square numbers are $1,4,9,5,3,0$. Therefore, $10^n+10^k+41$ is not a square number for any $n$ and $k$.
\end{proof}

\begin{lem}\label{lem:321}
The number $10^n+321$ is not a square number.
\end{lem}
\begin{proof}
If $N=10^{2k}+321$, $k\geqslant 3$, then
\[
(10^k)^2<10^{2k}+321<(10^k+1)^2.
\]
For $k=0,1,2$, it is easy to check that $N=10^{2k}+321$ is not a square number. Hence, $10^{2k}+321$ is not a square number.

If $N=10^{2k+1}+321$, then we consider the remainder of modular $13$.
\begin{equation}\label{eqn:10n-mod13}
\begin{aligned}
10^{2k+1} &\equiv 10,12,4\ (\mod 13),\\
10^{k} &\equiv 10, 9, 12, 3, 4, 1\ (\mod 13),\\
321 &\equiv 9\ (\mod 13),
\end{aligned}
\end{equation}
Thus $10^{2k+1}+321 \equiv 6,8,0\ (\mod 13)$.
The remainders by $13$ of the square numbers are $1,4,9,3,12,10,0$. Hence, if $N=10^{2k+1}+321$ is a square number, then it must be $10^{6m+5}+321$. It infers that $169|(10^{6m+5}+321)$. Thus, $169|(10^{6m+5}+152)$. Note that $9|10^{6m+5}+152$, we rewritten it as
\[
\begin{split}
10^{6m+5}+152&=(10^{6m+5}-1)+153=\underbrace{99\cdots 9}_{6m+5}+9\cdot 17=9\times \underbrace{11\cdots 1}_{6m+3}28\\
&=9\times 8\times 13\underbrace{88\cdots 8}_{6m}91.
\end{split}
\]
It is easy to check that $13|13\underbrace{88\cdots 8}_{6m}91$ for any $m\geqslant 0$ and $169|13\underbrace{88\cdots 8}_{72h}91$ for any $h\geqslant 1$.

Note that,
\[
13\underbrace{88\cdots 8}_{6m}91 \equiv 22\ (\mod 37),\quad\forall\ m.
\]
Thus,
\[
9\times 8\times 13\underbrace{88\cdots 8}_{6m}91+169\equiv 14\ (\mod 37).
\]

But the squares' remainders by $37$ are
\[
0, 1, 4, 9, 16, 25, 36, 12, 27, 7, 26, 10, 33, 21, 11, 3, 34, 30, 28.
\]
Therefore,
\[
10^{6m+5}+321=9\times 8\times 13\underbrace{88\cdots 8}_{6m}91+169
\]
is not a square number.
\end{proof}
%

%
%

\begin{lem}\label{lem:3}
The number $3\cdot 10^n+121$ is not a square number.
\end{lem}
\begin{proof}
\begin{equation}\label{eqn:10n-mod37}
10^n\equiv\begin{cases}
10\ (\mod 37), & n=3k+1,\\
26\ (\mod 37), & n=3k+2,\\
1 \ (\mod 37), &  n=3k.
\end{cases}
\end{equation}
Thus,
\[
3\cdot 10^n+121\equiv\begin{cases}
3 \ (\mod 37), & n=3k+1,\\
14\ (\mod 37), & n=3k+2,\\
13\ (\mod 37), &  n=3k.
\end{cases}
\]
While the square numbers' remainders under module $37$ are
\begin{equation}\label{eqn:b2-mod37}
0, 1, 4, 9, 16, 25, 36, 12, 27, 7, 26, 10, 33, 21, 11, 3, 34, 30, 28,
\end{equation}
we conclude that if $3\cdot 10^n+121$ is a square number, $n$ must be the form like $3k+1$.

If $k$ is even, then $N:=3\cdot 10^{3k+1}+121\equiv 8(\mod 11)$, it will not be a square number since
\[
a^2\equiv (0,1,4,9,5,3)\ \mod 11
\]
for any positive integer $a$. Then $k$ must be odd. Suppose $k=2h+1$, then $N=3\cdot 10^{6h+4}+121$. By the first equation in \eqref{eqn:10n-mod13}, we have
\[
3\cdot 10^{6h+4}+121\equiv 0\ (\mod 13).
\]
If $3\cdot 10^{6h+4}+121$ is a square number, then $169|(3\cdot 10^{6h+4}+121)$.

Note that
\[
10^{6h+4}\equiv(29,107,16,94,3,81,159,68,146,55,133,42,120)\ \mod 169.
\]
Then we have
\[
3\cdot 10^{6h+4}+121\equiv(39, 104, 0, 65, 130, 26, 91, 156, 52, 117, 13, 78, 143)\ \mod 169.
\]
It infers that $10^{6h+4}\equiv 16\ (\mod 169)$. And thus the number $N=3\cdot 10^{6h+4}+121$ takes the form
\[
N=3\cdot 10^{78k+16}+121\equiv 0\ (\mod 169).
\]
Since
\begin{equation}\label{eqn:10n-mod73}
10^n\equiv(10, 27, 51, 72, 63, 46, 22, 1)\ \mod 73,
\end{equation}
we have $10^{78}\equiv 46\ (\mod 73)$ and $10^{16}\equiv 1\ (\mod 73)$. Hence,
\[
10^{78k+16}\equiv 46^k\ (\mod 73)\equiv(46, 72, 27, 1)\ \mod 73.
\]
Therefore, we have
\[
3\cdot 10^{78k+16}+121\equiv (40, 45, 56, 51)\ \mod 73.
\]
But for the square number $b^2$,
\begin{equation}\label{eqn:a2-mod73}
\begin{split}
b^2\equiv &(0, 1, 4, 9, 16, 25, 36, 49, 64, 8, 27, 48, 71, 23, 50,\\
          &\ 6, 37, 70, 32, 69,35, 3, 46, 18, 65, 41, 19, 72, 54,\\
          &\ 38, 24, 12, 2, 67, 61, 57, 55)\ \mod 73.
\end{split}
\end{equation}
The list does not contain the numbers $40, 45, 56, 51$. Therefore, we conclude that $3\cdot 10^{78k+16}+121$ is not a square number.
\end{proof}

Generally, it is hard to prove the following result by using our method above. So here we use a different way to prove it.
\begin{lem}\label{lem:111}
The number like $10^m+10^n+10^k+121$ with $m\geqslant n\geqslant k$ is not a square number.
\end{lem}
\begin{proof}
First it is easy to see that there exist a positive integer $M>0$, such that for $k>M$, we have
\[
\Bigl|\sqrt{10^m+10^n+10^k+121}-\sqrt{10^m+10^n+10^k}\Bigr|<0.01.
\]
Similarly, for $n$ and $m$, we get a similar estimate.

Second, $\sqrt{10^m+10^n+10^k}$ is not a square number. 

Hence $\sqrt{10^m+10^n+10^k+121}$ is not a square number too.

At last, we only need to check finitely many numbers like the form 
\[
\sqrt{10^m+10^n+10^k+121}. 
\]
They are not square numbers. Here we omit the checking.
\end{proof}

It is difficult to prove the following lemma in our way. Lemma \ref{lem:4} is a special case of Lemma \ref{lem:111}. Nevertheless, we give the incomplete proof here for someone may be interesting in it.

\begin{lem}\label{lem:4}
The number like $2\cdot 10^n+10^m+121$ is not a square number.
\end{lem}
\begin{proof}
(Note that the proof is {\bf incomplete} here. Only some cases are proved.)

With Lemma \ref{lem:3}, we only need to prove the case $n\neq m$.
From the formula \eqref{eqn:10n-mod37}, we have
\[
2\cdot 10^n+10^m+121\equiv (3,19,31,35,14,26,22,1,13)\ \mod 37.
\]
Compare with \eqref{eqn:b2-mod37}, if $2\cdot 10^n+10^m+121$ is a square number, then
\[
2\cdot 10^n+10^m+121\equiv (3,26,1)\ \mod 37.
\]
Which correspond to the following three cases:
\[
\begin{aligned}
(a)\quad & n=3k+1,\ m=3h+1;\\
(b)\quad & n=3k+2,\ m=3h;\\
(c)\quad & n=3k,\ m=3h+2.\\
\end{aligned}
\]
{\bf(a)} For $N=2\cdot 10^{3k+1}+10^{3h+1}+121$,
\[
N\equiv\begin{cases}
8\ (\mod 11), & \text{if}\ k\ \text{and}\ h\ \text{are both even}, \quad(a1)\\
3\ (\mod 11), & \text{if}\ k\ \text{and}\ h\ \text{are both odd}, \quad(a2)\\
10\ (\mod 11), & \text{if}\ k\ \text{is even and}\ h\ \text{is odd}, \quad(a3)\\
1\ (\mod 11), & \text{if}\ k\ \text{is odd and}\ h\ \text{is even}. \quad(a4)\\
\end{cases}
\]
Note that
\[
a^2\equiv (0,1,4,9,5,3)\ \mod 11.
\]
So we only need to consider the two cases: (a2) and (a4).

(a4) For $k$ is odd and $h$ is even, with $2k+1$ substitute $k$ and $2h$ substitute $h$, we have
\[
\begin{split}
N&=2\cdot 10^{3(2k+1)+1}+10^{3(2h)+1}+121\\
&=2\cdot 10^{6k+4}+10^{6h+1}+121\equiv 7\ (\mod 13).
\end{split}
\]
It is not a square number since
\begin{equation}\label{eqn:a2-mod13}
a^2\equiv(0, 1, 4, 9, 3, 12, 10)\ (\mod 13).
\end{equation}

(a2) For $k$ and $h$ are both odd, then
\[
\begin{split}
N&=2\cdot 10^{3(2k+1)+1}+10^{3(2h+1)+1}+121\\
&=2\cdot 10^{6k+4}+10^{6h+4}+121\\
\end{split}
\]
If $h>2k\geqslant 1$, it is easy to see that
\[
(10^{3h+2})^2=10^{6h+4}<N<(10^{3h+2}+1)^2.
\]
If $\frac{4}{3}k<h\leqslant 2k$, let $\ell=6k-3h+2$, then we have
\[
(10^{3h+2}+10^{\ell}-1)^2<N<(10^{3h+2}+10^{\ell})^2.
\]

If $k<h\leqslant\frac{4}{3}k$, then $N<(10^{3h+2}+10^{\ell}-1)^2$.
Suppose $N=(10^{3h+2}+10^{\ell}-q)^2$, then we have
\[
(10^{3h+2}+10^{\ell})^2-(10^{\ell})^2+121=(10^{3h+2}+10^{\ell})^2-2q(10^{3h+2}+10^{\ell})+q^2,
\]
which infers that
\[
(10^{\ell})^2-11^2=q\cdot\Bigl[2(10^{3h+2}+10^{\ell})-q\Bigr].
\]
Note that
\[
\begin{split}
N&=2\cdot 10^{3(2k+1)+1}+10^{3(2h+1)+1}+121\\
&=2\cdot 10^{6k+4}+10^{6h+4}+121\\
&\equiv 0\ (\mod 13).
\end{split}
\]
We have $13|(10^{3h+2}+10^{\ell}-q)$. Let $q=10^{3h+2}+10^{\ell}-13M$. Then
\[
(10^{3h+2}+10^{\ell}-11)(10^{3h+2}+10^{\ell}+11)=(10^{3h+2}+10^{\ell}-13M)(10^{3h+2}+10^{\ell}+13M).
\]
But, $\frac{4}{3}k<h$ infers that $3h\geqslant 4k+1$. Thus $3h+2>2\ell$, and
\[
10^{3h+2}+10^{\ell}+13M > 10^{2\ell}+10^{\ell}+13M > (10^{\ell})^2-11^2,
\]
It is a contradiction.

That is, we have proved the number like $10^{3h+1}+2\cdot 10^{3k+1}+121$ with $h>k$ is not a square number.

{\bf Now we assume $k>h$.} The case $h=k$ has been solved in Lemma \ref{lem:3}.

{\bf But it is hard to prove.}

\medskip

{\bf(b)} For $N=2\cdot 10^{3k+2}+10^{3h}+121$,
\[
N\equiv\begin{cases}
8\ (\mod 11), & \text{if}\ k\ \text{and}\ h\ \text{are odd}, \quad(b1)\\
3\ (\mod 11), & \text{if}\ k\ \text{and}\ h\ \text{are even}, \quad(b2)\\
1\ (\mod 11), & \text{if}\ k\ \text{is even and}\ h\ \text{is odd}, \quad(b3)\\
10\ (\mod 11), & \text{if}\ k\ \text{is odd and}\ h\ \text{is even}. \quad(b4)\\
\end{cases}
\]
Thus, we only need to consider the two cases: (b2) and (b3).

(b3) For $k$ even and $h$ odd, we have
\[
\begin{split}
N&= 2\cdot 10^{3(2k)+2}+10^{3(2h+1)}+121\\
&=2\cdot 10^{6k+2}+10^{6h+3}+121\\
&\equiv 8\ (\mod 13).
\end{split}
\]
So, it is not a square number since \eqref{eqn:a2-mod13}.

(b2) For $k$ and $h$ are both even,
\[
N=2\cdot 10^{3(2k)+2}+10^{3(2h)}+121=2\cdot 10^{6k+2}+10^{6h}+121.
\]



Note that $10^n\equiv(3, 2, 6, 4, 5, 1)\ \mod 7$, we have
\[
\begin{split}
N&=2\cdot 10^{6k+2}+10^{6h}+121\\
&\equiv 2\cdot 2+1+2\ (\mod 7)\\
&\equiv 0\ (\mod 7).
\end{split}
\]
If $N$ is a square number, then $49|N$. Since $N$ ends with digit $1$, we have
\[
49\times(10M+9)=N.
\]

Since $N=2\cdot 10^{6k+2}+10^{6h}+121\equiv 8(\mod 13)$ and $49\equiv 10(\mod 13)$, we have $10M+9\equiv 6(\mod 13)$. Thus $10M+9$ will not be a square number by \eqref{eqn:a2-mod13}.

Therefore, $N=2\cdot 10^{6k+2}+10^{6h}+121$ is not a square number.

\medskip

\noindent{\bf (c)} For $N=2\cdot 10^{3k}+10^{3h+2}+121$, {\bf it is hard to prove.}

According \eqref{eqn:10n-mod13}, we have
\[
N\equiv\begin{cases}
6\ (\mod 13), & \text{if}\ k\ \text{and}\ h\ \text{are odd}, \quad(c1)\\
2\ (\mod 13), & \text{if}\ k\ \text{and}\ h\ \text{are even}, \quad(c2)\\
10\ (\mod 13), & \text{if}\ k\ \text{is even and}\ h\ \text{is odd}, \quad(c3)\\
11\ (\mod 13), & \text{if}\ k\ \text{is odd and}\ h\ \text{is even}. \quad(c4)\\
\end{cases}
\]
Since $a^2\equiv(0,1,4,9,3,12,10)\mod 13$, we only need to consider the (c3) case.

Thus
\[
N=2\cdot 10^{6k}+10^{6h+5}+121\equiv 2(\mod 7),
\]
which infers that $N=(7K\pm 3)^2$. And $N\equiv 1(\mod 11)$.

We consider mod 73 for $N$. And according \eqref{eqn:10n-mod73}, we list the following cases. 

%
%

Since $10^n\equiv(10,100,1)\mod 999$, we have
\[
N=2\cdot 10^{6k}+10^{6h+5}+121\equiv 223\ (\mod 999).
\]
Similarly, there are only following four kind numbers with remainder $223$.
\[
\begin{aligned}
(999k+149)^2\equiv 223\ (\mod 999),\\
(999k+445)^2\equiv 223\ (\mod 999),\\
(999k+554)^2\equiv 223\ (\mod 999),\\
(999k+850)^2\equiv 223\ (\mod 999).\\
\end{aligned}
\]
(I was inspired from these equations. By computing I found an interesting formula and prove it (see \cite{Xu}).)
\end{proof}


\noindent{\bf Acknowledgments :}
we express our gratitude to Professor Rongzheng Jiao of Yangzhou University.


\bigskip

\noindent Haifeng Xu\\
School of Mathematical Sciences\\
Yangzhou University\\
Jiangsu China 225002\\
hfxu@yzu.edu.cn\\

%


\end{document}